\documentclass[12pt,a4paper]{amsart}
\usepackage{amsfonts,amsmath,amsthm,amssymb, amscd}
\usepackage[top=2cm, bottom=4cm, left=3cm, right=1cm]{geometry}

\usepackage{xcolor}

\newcommand{\vn}{\varnothing}

 \DeclareMathOperator{\im}{Im}

\DeclareMathOperator{\Ker}{Ker}
\DeclareMathOperator{\charak}{char}
\DeclareMathOperator{\Der}{Der}

\DeclareMathOperator{\IDer}{IDer}

\numberwithin{equation}{section}
\newtheorem{theorem}{Theorem}
\newtheorem{lemma}[theorem]{Lemma}
\newtheorem{remark}[theorem]{Remark}
\newtheorem{corollary}[theorem]{Corollary}
\newtheorem{proposition}[theorem]{Proposition}

\begin{document}

\title[Engel condition]{Torsion subgroups, solvability and the Engel condition in associative rings}
%\thanks{Corresponding Author: \quad V.~Bovdi\\
%The research  was supported by  the UAEU  UPAR  grant  G00002160}

\author[Artemovych,  Bovdi]{Orest Artemovych,  Victor Bovdi}
\address{Department of Applied   Mathematics,  Cracow University of Technology, Cracow,  Poland}
\email{artemo@usk.pk.edu.pl}

\address{UAEU, United Arab Emirates}
\email{vbovdi@gmail.com}


\maketitle

%%%%%%%%%%%%%%%%%%%%%%%%%%%%%%%%%%%%%%%%%%%%%%%%%%%%%%%%

\begin{abstract}
The connections between the properties of associative rings  that are Lie-solvable (Engel, n-Engel, locally finite, respectively) and the properties of their adjoin  subgroups are investigated.
\end{abstract}

%================================

\section{ Introduction}

%===================================

Let $R=(R,+,\cdot )$ be an  associative ring (not necessary with
unity). The set of all elements of $R$ forms a semigroup with
respect  to the circle operation ``$\circ$"{} defined by the rule
$a\circ b=a+b+a\cdot b$ for each  $a,b\in R$. The set
\[
R^{\circ}=\{
a\in R\mid a\circ b=0=b\circ a\quad  \text{for some}\quad  b\in R\}
\]
is a group (so-called {\it the adjoint group}\ of $R$). If $R$ has
unity and $U(R)$ is the unit group of $R$, then 
\[
R^{\circ}\ni
a\mapsto 1+a\in U(R)
\]
 is a group isomorphism. If $R=R^\circ$, then  $R$ is called {\it radical}.

We study properties of associative rings and their adjoint groups which are connected  with solvability,  Engel conditions and periodicity.

We always assume that $p$ is a prime number, $\mathbb N$ is the set of positive integers,  $\mathbb F$ is a field,
${\mathbb Z}_n$ is the ring of integers modulo $n$. Let $n,k\in\mathbb{N}$, $m\in\mathbb{Z}$ and let  $x,g\in R$. We introduce the following notation:

$\mu_n(x)=\sum\limits_{k=1}^n\binom{n}{k}x^k$,

$[x,{}_1g]=[x,g]=x\cdot g-g\cdot x,\qquad  [x,{}_{n+1}g]=[[x,{}_{n}g],g]$,

$C_R(g)=\{ x\in R\mid xg=gx\}$  is the centralizer of $g$ in $R$,

 $J(R)$ is  the Jacobson radical of $R$,

 $N(R)$  is the set of all nilpotent elements  of $R$,

 $g^{(m)}$ is the $m$-th power of $g$ in $R^\circ$,

 $F(R)=\{ x\in R\mid x\ \text{is of finite order in}\ R^+\}$ is the torsion part of the additive group $R^+$,

 %$F_k(R)=\{ x\in F(R)\mid kx=0\}$ the torsion $k$-part of  $R^+$,

 $\charak R$ is the characteristic of $R$,

 $Z(R)$ is the center of $R$,

 ${\mathbb P}(R)$ is the prime radical of $R$ (i.e., the intersection of all prime ideals of
$R$),

$N^*(R)$ is the nil radical of $R$ (i.e., the sum of all nil ideals),

$N_r(R)$ is  the sum of all nil right ideals of $R$ (moreover, $N_r(R)$ is  the sum of all nil left ideals of $R$ and, therefore, $N_r(R)$ is a two-sided ideal of $R$),

$[A,B]$ is the additive subgroup  of $R^+$ generated by all $[a,b]$,
where $a\in A, b\in B$ and $A,B\subseteq R$,

$C(R)$  is the commutator ideal of $R$ (i.e., an ideal of $R$
generated by all $[g,x]$),

$D=\Der R$ is the set of all derivations of $R$,

$\Delta (R)$  is the  ideal of $R$ generated by all $\delta (R)$, where $\vn \neq \Delta \subseteq \Der R$ and $\delta \in \Delta$,

$\gamma_1R=[R,R]$ and $\gamma_{n+1}R=[\gamma_{n}R,R]$,

$\delta_0R=R$ and $\delta_{n+1}R=[\delta_nR,\delta_nR]$,

${\langle X\rangle}_{\rm rg}$ is a   subring of $R$ generated
by $X\subseteq R$ (if $X=\vn$, then ${\langle X\rangle}_{\rm
rg}=0$).

If $\Delta =(\Delta, +[-,-])$ is a Lie ring, then
\[
\gamma_1\Delta:=\Delta, \ \ldots, \gamma_{k+1}\Delta:=[\gamma_k\Delta, \Delta ],\ldots  \qquad (k\in\mathbb{N}).
\]
and $\Delta^{(1)}:=\Delta,\ldots,\quad  \Delta^{(n+1)}:=[\Delta^{(n)},\Delta^{(n)}]$.

Let  $\tau (G)$ be the set of all torsion elements of a group $G$.
Recall that a ring $R$ is called:
\begin{itemize}
\item[$\bullet$] {\it nil}  if each $x\in R$ is nilpotent, i.e., there exists $n=n(x)\in\mathbb{N}$ such that $x^n=0$; if there exists  $n\in\mathbb{N}$ such that $x^n=0$ for any $x\in N(R)$, then $R$ is    {\it of bounded index of nilpotency} ({\it of bounded index} for short),

\item[$\bullet$] {\it local} if $R\ni 1$ and $R/J(R)$ is a simple ring,

\item[$\bullet$] {\it right Artinian} in case for each  ascending chain
\begin{equation*}
I_1\supseteq  I_2\supseteq  \cdots \supseteq  I_n\supseteq  \cdots
 \end{equation*}
of right ideals $I_j$ of $R$\ $(j=1,2,\ldots )$, there exists $n\in\mathbb{N}$ such that  $I_{n+1}=I_n$,

\item[$\bullet$] {\it right Noetherian} in case for each  descending chain
\begin{equation*}
I_1\subseteq  I_2\subseteq  \cdots \subseteq  I_n\subseteq  \cdots
 \end{equation*}
of right ideals $I_j$ of $R$\ $(j\in\mathbb{N})$, there exists $n\in\mathbb{N}$ such that $I_{n+1}=I_n$,

\item[$\bullet$] {\it semilocal} if $R/J(R)$ is a left
Artinian ring,

\item[$\bullet$] {\it right Goldie } if it has no infinite direct sum of left ideals and has the ascending chain condition on right annihilators,

\item[$\bullet$] {\it Lie nilpotent} of class $n$
if $n$ is a minimal positive integer such that $\gamma_{n+1}R=0$,

\item[$\bullet$] {\it locally nilpotent} if each  its finitely generated subring is nilpotent,

\item[$\bullet$] {\it locally Lie  nilpotent} if each   finitely generated subring of $R$ is Lie nilpotent,

\item[$\bullet$] {\it Lie soluble} of length at most $n$ if $\delta_nR=0$,

\item[$\bullet$] {\it Lie metabelian} if $\delta_2R=0$,

\item[$\bullet$] {\it Lie centrally metabelian} if $\delta_2R\subseteq Z(R)$,

\item[$\bullet$] {\it Engel} (or equivalently $R$ satisfies {\it  the Engel
condition}) if, for each $x,y\in R$, there exists $n=n(x,y)\in\mathbb{N}$ such that $[x,{}_ny]=0$,

\item[$\bullet$] {\it $n$-Engel} (or equivalently $R$ satisfies {\it  the $n$-Engel
condition} or $R$ is {\it bounded Engel}) if  $[x,{}_ny]=0$ for any $x,y\in R$,

\item[$\bullet$] {\it semiprime} if it has no nonzero nilpotent ideals,

\item[$\bullet$]  {\it prime} if the product of each two nonzero ideals  is nonzero,

\item[$\bullet$] {\it simple} if  $R^2\neq 0$ and   $0$, $R$ are the only ideals of   $R$,

\item[$\bullet$] {\it reduced} if $N(R)=0$,

\item[$\bullet$] {\it $2$-primal} if ${\mathbb P}(R)=N(R)$,

\item[$\bullet$] {\it right quasi-duo} if each maximal right ideal is two-sided,

\item[$\bullet$] {\it of stable range $1$} if $R\ni 1$ and, for each $a,b,x,y\in R$ with $ax+by=1$, there exists $h\in R$ such that $a+bh\in U(R)$,

\item[$\bullet$] {\it abelian} if its all idempotents are central,

\item[$\bullet$]  {\it   $\pi$-regular}  if, for each  $a\in R$, there exists $b\in R$ such that $a^n=a^nba^n$  for some $n\in\mathbb{N}$,
    \item[$\bullet$] {\it right weakly $\pi$-regular} if, for each $a\in R$, there exists  $n=n(a)\in {\mathbb N}$ such that $a^nR=(a^nR)^2$.
\end{itemize}

We shall use freely the following well known facts: Right Artinian rings are $\pi$-regular and $\pi$-regular rings are weakly $\pi$-regular. An associative ring $R$ is Lie nilpotent (Engel, respectively) if and only if  $R^L$ is nilpotent (Engel, respectively). Each (Lie or associative) locally nilpotent ring is Engel. Each  Lie nilpotent ring of nilpotency class $n$ is $n$-Engel. A radical ring $R$ is Lie nilpotent   if and only if its adjoint group $R^\circ$ is nilpotent \cite{Jennings47}.

In certain papers (see, for example,  \cite{Amberg_Sysak_III, Bokut_Lvov_Kharchenko, Sharma_Srivastava85, Zalesskii_Smirnov_82} and others) { by many authors was investigated properties of Lie solvable rings and its relations with    groups.  }  Each radical ring $R$ with solvable adjoint group $R^\circ$ is  Lie solvable \cite[Theorem A]{Amberg_Sysak_I}. As a consequence, if $R$ is a semilocal ring with the solvable unit group $U(R)$, then $R$ is Lie  solvable.

Each $2$-torsion-free Lie solvable ring  $R$ has a nilpotent ideal $I$ such that  $R/I$ is Lie centre-by-metabelian (and so $R^\circ$ is solvable) (see \cite{Zalesskii_Smirnov_82} and \cite{Sharma_Srivastava85}).  There exists a Lie center-metabelian (and so it is Lie solvable) total $(2\times 2)$-matrices ring $M_2(R)$ over  an infinite commutative domain $R$ of characteristic $2$, but its adjoint group $M_2(R)^\circ$ is non-solvable. The  group of units  $U(R)$ of a Lie metabelian unitary ring $R$ is metabelian (see \cite{Sharma_Srivastava92} and \cite[Theorem 1]{Krasilnikov}).

Our  result is  the following.

\begin{theorem} \label{ZAZ} Let  $R$ be  a  ring. The following statements hold:
\begin{itemize}
\item[$(i)$] if $R$ is Lie nilpotent, then the adjoint group $R^\circ$ is nilpotent-by-abelian;
\item[$(ii)$] {if $R^\circ$ is solvable-by-finite, then $[J(R),R]\subseteq {\mathbb P}(R)$ and $J(R)C(R)\subseteq {\mathbb P}(R)$;}
\item[$(iii)$] if $R$ is $2$-torsion-free Lie solvable, then:
\begin{itemize}
\item[$(a)$]
$C(R)\subseteq {\mathbb P}(R)=N(R)$ (i.e., $R$ is $2$-primal and right quasi-duo), $N(R)$ is the locally nilpotent ideal of $R$  and $R^\circ$ is (locally nilpotent)-by-abelian. Moreover, if $R$ has unity, then it is abelian;
\item[$(b)$] if $R^\circ$ is torsion, then it is locally  nilpotent;
\item[$(c)$] if $N(R)^+$ is torsion-free divisible, then $R^\circ$ is nilpotent-by-abelian.
\end{itemize}
\end{itemize}
\end{theorem}

In the local case we obtain the following.

\begin{theorem} \label{LSS} Let $R$ be a local ring. The following statements hold:
\begin{itemize} \item[$(i)$] if the unit group $U(R)$ is solvable, then:
\begin{itemize}
\item[$(a)$] $C(R)\subseteq L(R)=N(R)$ is a locally nilpotent ideal of $R$, where $L(R)$ is the Levitzki radical of $R$;
\item[$(b)$] $R$ is Lie solvable;
\end{itemize}
\item[$(ii)$] if $R$ is a Lie solvable $\mathbb Q$-algebra, then $R^\circ$ is nilpotent-by-abelian.
\end{itemize}
\end{theorem}
Recall that the Levitzki radical $L(R)$ of $R$ is its unique maximal locally nilpotent ideal.

An additive map $\delta: R\to R$ is called {\it a derivation} of $R$ if $\delta (ab)=\delta (a)b+a\delta (b)$ for all $a,b\in R$. The set $D$ of all derivations of $R$ is a Lie ring.  Properties of  a ring $R$ which induced by the Engel condition    of the  derivation ring $D$ gives the following.

\begin{theorem}\label{GGX} Let $R$ be a ring. If  $\vn \neq \Delta \subseteq D$,  then
\begin{itemize}
\item[$(i)$]   if $\Delta$ is a nilpotent Lie ring and $R\Delta \subseteq \Delta$ modulo $C(R)$, then $\Delta (R)^m\subseteq C(R)$ (in particular, if\ $D$ is Lie nilpotent, then $D(R)^m\subseteq C(R)$) for some $m\in \mathbb{N}$;
\item[$(ii)$]  {
if $\Delta$ is a nilpotent  Lie ring of the nilpotent length $n$ and $R\Delta \subseteq \Delta$ modulo $C(R)$, then $d^{n}(R)d(R)\subseteq C(R)$ for some $d\in \Delta^{(n-1)}$;
  \item[$(iii)$]  if $\Delta$ is an Engel Lie ring and  $R\Delta \subseteq \Delta$ modulo $C(R)$, then, for each $a\in R$ and $\delta \in \Delta$,  there exists  $n=n(\delta, a)\in {\mathbb N}$ such that  $\delta^n(a)\delta (R)\subseteq C(R)$.
      }
\end{itemize}
\end{theorem}

If the unit group $U(R)$ of a semilocal ring $R$ is $m$-Engel, then $U(R)$ is locally nilpotent and, furthermore, $R$ is $n$-Engel provided that $R$ is generated by $U(R)$ (bibliography in this way see in \cite{Amberg_Sysak_IV, Picelli}). Moreover,  a local ring $R$ is Lie nilpotent if and only if  $U(R)$ is nilpotent \cite{Stolz} and that is this case the classes of nilpotency of both structures coincide. { We prove the next.}

\begin{theorem} \label{FLE1}
Let $R$ be a  $n$-Engel local ring. If $F(R)=0$, then $R$ is Lie nilpotent.
\end{theorem}

A ring  $R$ is called {\it locally finite} if each finite subset of $R$ generates a finite semigroup multiplicatively.   The class of locally finite rings is closed under formation of
subrings, homo\-mor\-phic images and direct sums (see \cite[Proposition
2.1]{Huh_Kim_Lee04}). A finite subset of a locally finite ring generates a finite
subring (not necessary with unity) \cite[Theorem 2.2]{Huh_Kim_Lee04}  and a
locally finite ring is strongly $\pi$-regular \cite[Lemma
2.4(ii)]{Huh_Kim_Lee04}. Recall that a ring $R$ is called {\it strongly
$\pi$-regular} if, for each  $a\in R$, there exist  $n=n(a)\in {\mathbb N}$ and $b\in R$ such that $a^n=a^{n+1}b$.
A ring $R$ is strongly $\pi$-regular \cite{Azumaya} if and only if it satisfies
the descending chain condition on principal right ideals of the
form
\[
aR\supseteq  a^2R\supseteq  \cdots \supseteq   a^nR\supseteq   \cdots \qquad (\forall a\in R).
\]
Local rings with  the nil Jacobson radical and  semilocal rings  with the nil Jacobson radical of bounded index   are  strongly $\pi$-regular \cite[Lemma 3.1 and Corollary 3.3]{Huh_Kim_Lee04}. The Jacobson radical $J(R)$ of a locally finite ring $R$ is  a locally nilpotent ring (in view of \cite[Corollaries 1 and 4]{Amberg_Dickenschied_Sysak} and Lemma \ref{LF1}$(ii)$). We precise \cite[Propositions 2.5,  2.10  and 2.11]{Huh_Kim_Lee04} in the following

\begin{proposition} \label{LF3} Let $R$ be a $2$-torsion-free locally finite ring with unity. The following state\-ments hold:
\begin{itemize}
\item[$(i)$] each  prime ideal of  $R$  is maximal as a right ideal (and so $R$ is $\pi$-regular);
\item[$(ii)$] $R$ is an abelian exchange ring of stable range $1$, $C(R)\subseteq {\mathbb P}(R)=N(R)=J(R)$ and $R/J(R)$ is a subdirect product of locally finite fields (so each element of  $R$ is a sum of a unit  and a central element).
\end{itemize}
\end{proposition}

Each  absolute field (i.e.,  a field in which  each   nonzero element is a
root of $1$) is a locally finite ring. In a locally finite field
$\mathbb F$ every finite subset $X\subseteq {\mathbb F}$ generates
a finite subfield. Since the unit group $U({\langle X\rangle
}_{\rm rg})$ is cyclic, we deduce that a locally finite field is
absolute. A locally finite right Noetherian ring is right Artinian  (see Proposition \ref{PiPi1}).

Rings with torsion adjoint groups were intensively studied in \cite{Amberg_Dickenschied_Sysak, Herstein70, Herstein78, Herstein80, Krempa95, Krempa99, Shlyafer, Sysak01} and others. It is well known \cite[Theorem 8]{Herstein78} that  a division ring $D$ with the torsion multiplicative  group $D^*$ is commu\-ta\-ti\-ve.  Moreover,  a torsion normal subgroup of the multiplicative group $U(D)$ of a skew field $D$ is central \cite[Lemma 10]{Zalesskii65}. Each  torsion subgroup of a linear group  over a field is locally finite by classical results of  W.~Burnside and I.~Schur.  A torsion subgroup of the unit group $U(R)$ of a unitary $PI$-ring is locally finite  by results of C.~Procesi and  A.I.~Shirshov.  Each locally finite subgroup  of the adjoint group $R^\circ$ of a radical ring $R$ is locally nilpotent \cite[Corollary 1]{Amberg_Dickenschied_Sysak}.

We have the following.

\begin{proposition} \label{XY0} Let $R$ be a ring such that  $R^\circ$ is torsion and $F(R)=0$. The following statements hold:
\begin{itemize}
\item[$(i)$]  $R$ is commutative or without zero-divisors,
\item[$(ii)$] if $R$ is prime with unity, then  $R$  is a domain such that  $J(R)=0$ and the    unit group $U(R)$  is finite of  one of the following types:
\begin{itemize}
\item[$(a)$]  $U(R)$ is a cyclic group of order $2$ such that   $\langle U(R) \rangle_{rg}\cong {\mathbb Z}$;
\item[$(b)$] $U(R)$ is a cyclic group of order $4$ such that  $\langle U(R) \rangle_{rg}\cong{\mathbb Z}[i]$ is the ring of Gaussian integers;
\item[$(c)$]$U(R)$ is a cyclic group of order $6$ such that  $\langle U(R) \rangle_{rg}\cong{\mathbb Z}[\zeta_3]$ is the subring of integer elements of the Eisenstein field ${\mathbb Q}[i\sqrt{3}]$;
\item[$(d)$]  $U(R)$ is the quaternion group  of order $8$ such that  $\langle U(R) \rangle_{rg}\cong{\mathbb Z}[i,j]$ is the ring of quaternions with integer coefficients;

\item[$(e)$] $U(R)=\langle a,b\mid a^3=b^2=(ab)^2\rangle$ is the dicyclic group of order $12$ such that  $\langle U(R) \rangle_{rg}\cong  {\mathbb Z}\cdot 1+{\mathbb Z}\cdot \alpha +{\mathbb Z}\cdot \beta +{\mathbb Z}\cdot \gamma$ is the ring with the following Cayley table of multiplication:
\[    
\begin{array}{||l||r|r|r||}
\hline
{\cdot}& \alpha & \beta &\gamma\\
\hline \hline
\alpha & -1& \gamma & -\beta  \\
\hline
\beta & -\alpha -\gamma  & -1-\beta  &  \alpha   \\
\hline
\gamma & 1+\beta & -\alpha-\gamma &-1 \\
\hline
\end{array}
\]
\item[$(f)$]  $U(R)=\langle a,b \mid a^3=b^3=(ab)^2=1\rangle$ is the binary tetrahedral  group of order $24$ such that   $\langle U(R) \rangle_{rg}$ is a subring of the skew field  ${\mathbb Q}(i,j)$ of quaternions which is generated by $i,j$ and $\frac{1+i+j+k}{2}$.
\end{itemize}
\end{itemize}
\end{proposition}

 Any unexplained terminology is standard as in  \cite{Faith2, Robinson}.

%==========================

\section{Semiprime rings}

%=========================

The unit group $U(R)$ of a unitary $n$-Engel ring $R$ is $m$-Engel for some  $m=m(n)\in {\mathbb N}$ depending on $n$  (see  \cite[Corollary 1]{Riley_Wilson} and  \cite[Corollary]{Shalev}).  The adjoint group $R^\circ$ of a radical ring $R$ is $n$-Engel if and only if $R$ is an $m$-Engel ring for some  $m=m(n)\in {\mathbb N}$   \cite[Main Theorem]{Amberg_Sysak_I}. Each  $n$-Engel Lie algebra is locally nilpotent and each  $n$-Engel Lie algebra over a field of characteristic zero is nilpotent \cite{Zelmanov87, Zelmanov89, Zelmanov90}.

If a group $G$ contains a non-trivial $p$-element and the unit group $U({\mathbb F}[G])$ of the group algebra ${\mathbb F}[G]$ is bounded Engel, then  ${\mathbb F}[G]$ is bounded Engel  \cite{Bovdi06} (see also \cite{BKh91, Bovdi20}).
An unitary associative bounded Engel algebra $A$ over a field of prime characteristic has the bounded Engel group $U(A)$ \cite{Shalev}, which is  locally nilpotent  \cite[Remark]{Shalev}. In the case of zero characteristic,  $U(A)$  is nilpotent and $A$ is Lie nilpotent (see \cite{Kemer} and \cite{Riley_Wilson}).   Each  bounded Engel subgroup of the adjoint group $R^\circ$ of a radical ring $R$ is locally nilpotent  \cite[Corollary 1]{Amberg_Dickenschied_Sysak}.

It is known that
$N_r(R)\subseteq  R^\circ$, \quad $N^*(R)\subseteq J(R)\subseteq R^\circ$,
\[
{\mathbb P}(R)\subseteq L(R)\subseteq N^*(R)\subseteq N_r(R)\subseteq N(R).
\]
If $R$ is an $n$-Engel ring, then  its commutator ideal $C(R)$ is nil (for example, see  \cite[Application 2]{Herstein63}) and,  additionally,   $C(R)\subseteq L(R)$ \cite[Lemma 3.1]{Amberg_Sysak_I}.

We use the following.

\begin{lemma} \cite[Theorem 1]{Lanski97} \label{la1}
Let $L\not=0$ be a  left ideal of a prime ring $R$. Let  $D$ be the Lie ring of all derivations of  $R$.  Let $k,n\in \mathbb{N}$ and let $0\neq \delta \in D$. If
 \begin{equation} \label{PIP}
 [\delta (x^k),{}_nx^k]=0,\qquad (\forall x\in L)
 \end{equation}
then $R$ is commutative.
 \end{lemma}
As a consequence,  we have the next.
\begin{corollary} \label{PP} Let   $R$ be an $n$-Engel ring.  The following statements hold:
\begin{itemize}
 \item[$(i)$]  $C(R)\subseteq {\mathbb P}(R)=N(R)\subseteq J(R)$ (i.e., $R$ is $2$-primal and right quasi-duo);
 \item[$(ii)$] if $R\ni 1$, then $R$ is abelian;
 \item[$(iii)$] the adjoint group $R^\circ$ is  (locally nilpotent)-by-abelian and $N(R)$ is a locally nilpotent ideal of $R$;
 \item[$(iv)$] if $F(N(R))=0$, then $R^\circ$ is nilpotent-by-abelian.
 \end{itemize}
 \end{corollary}
\begin{proof} $(i)$ The quotient ring $R/P$ is $n$-Engel for each prime ideal  $P$ of $R$ and each inner derivation $\delta$ of $R/P$ satisfies $(\ref{PIP})$, so we conclude that $R/P$ is commutative by Lemma \ref{la1}. That yields  $C(R)\subseteq {\mathbb P}(R)=N(R)$  and  $N(R)$ is locally nilpotent as  a $PI$-ring. Thus,  each maximal right ideal of $R$ is two-sided.

$(ii)$ If  $R$ has unity, then it is abelian by \cite[3.20]{Tuganbaev02}.

$(iii)$ The set $N(R)$ is an ideal in view of  the part $(i)$, $N(R)^\circ$ is $m$-Engel for some $m\in \mathbb{N}$ \cite[Main Theorem]{Amberg_Sysak_I} and $N(R)^\circ$ is locally nilpotent \cite[Lemma 2.2]{Amberg_Sysak_I}. Then $N(R)$ is locally nilpotent by \cite[Lemma 3]{Amberg_Dickenschied_Sysak}.

$(iv)$ If $F(N(R))=0$, then $N(R)^\circ$ is torsion-free. Since $N(R)^\circ$ is $m$-Engel for some $m\in\mathbb{N}$ \cite[Main Theorem]{Amberg_Sysak_I}, it is nilpotent  \cite{Zelmanov89}.
\end{proof}

\begin{lemma}(see \cite[Theorem 3]{Chuang_Lin}) \label{cl2}
Let $R$ be a ring with $N_r(R)=0$. If, for given $x,y\in R$ there exist  positive integers $m=m(x,y)$, $n=n(x,y)$ and $k=k(x,y)$ such that
\begin{equation} \label{KOKO}
 [x^m,{}_ky^n]=0,
 \end{equation} then $R$ is commutative.
\end{lemma}
An unitary associative $PI$-algebra $R$ with the Engel  condition over a field of any charac\-te\-ris\-tic has the nil commutator ideal by \cite[Proposition 2.3]{Riley00}. Furthermore, each  nil ring is Engel by \cite[Proposition 4.2]{Sysak10}, but there exist nil rings $R$ (which  also are algebras over arbitrary fields)  such that their adjoint groups $R^\circ$ are not Engel \cite[Theorem 1.5]{Smoktunowicz}. Inasmuch as  an Engel ring $R$ satisfies  (\ref{KOKO}), we obtain the following.

\begin{corollary} \label{19} Let  $R$ be an Engel  ring. The following statements hold:
 \begin{itemize}
\item[$(i)$] $C(R)\subseteq N_r(R)=N(R)\subseteq J(R)$, $N(R)$ is an ideal of $R$ (i.e., $R$ is right (left) quasi-duo) and $R^\circ$ is Engel-by-abelian;
\item[$(ii)$] if $R$ is  locally Lie nilpotent, then the adjoint group  $R^\circ$ is (locally nilpotent)-by-abelian;
\item[$(iii)$]    if $R$ has unity, then it  is abelian.
\end{itemize}
 \end{corollary}
 \begin{proof}
{ The part $(i)$ follows from Lemma \ref{cl2}. Assume that $R$ is locally Lie nilpotent. If $x,y\in R$, then the subring $\langle x,y\rangle_{\rm rg}$ is Lie nilpotent and so there exist $n=n(x,y)\in {\mathbb N}$ such that $[x,{}_ny]=0$. This means that $R$ is Engel, the adjoint group $N(R)^\circ$ is locally nilpotent in view of \cite[Main Theorem]{Amberg_Sysak_II} and so $R^\circ$ is (locally nilpotent)-by-abelian. }
All maximal right ideals in unitary Engel ring $R$ are two-sided so $R$ is  abelian by \cite[3.20]{Tuganbaev02}.\end{proof}

{ Every domain of characteristic $0$ that is Engel (as a Lie ring) is commutative \cite[Theorem 4]{Bell_Klein}. If the unit group $U(D)$ is $m$-Engel, then a division  ring $D$ is commutative by \cite[Lemma 4.1]{Amberg_Sysak_IV}. We obtain an affirmative answer on \cite[Question1.2]{Evstafev}.}

\begin{proposition} \label{D1D} Each Engel division ring  is commutative.
\end{proposition}

\begin{proof} The assertion holds from Corollary \ref{19}.
\end{proof}

Our next result confirm  a conjecture of \cite[Hypothesis 1.1]{Evstafev}.

\begin{proposition}
An   Engel adjoint group $R^\circ$ of  a right Artinian ring $R$ is nilpotent.
\end{proposition}

\begin{proof}
The Jacobson radical $J(R)$ is a nilpotent ideal of $R$, so  $J(R)^\circ$ is a
nilpotent group. We can assume that $J(R)^2=0$.
Since $R$ is a right Noetherian \cite[Theorem
18.3]{Faith2},  $J(R)$ is a finite direct sum of minimal ideals of $R$ and  $J(R)\subseteq Z(R)$ by \cite[Lemma
2.1]{Evstafev}. Thus $R^\circ$ is a nilpotent group by a well-known Ph.~Hall's
Theorem. \end{proof}

%==========================

\section{Torsion subgroups}

%=========================

It is well known the following.

\begin{lemma} \label{krempa} Let $R$ be a ring with $1$. The following statements hold:
\begin{itemize}
\item[$(i)$]  $U(R)$ is a torsion group if and only if the Jacobson radical $J(R)$ is nil with the torsion additive group $J(R)^+$ and $U(R/J(R))$ is torsion \cite[Lemma 1.1]{Krempa95};
\item[$(ii)$] if $F(R)=0$ and $U(R)$ is torsion, then $J(R)=0$ and $R$ is reduced \cite[Corollary 1.2]{Krempa95}; \item[$(iii)$]  if $F(R)=0$, then $U(R)$ is torsion if and only if $U(R)$ is locally finite (\cite[Theorem 3.3]{Krempa95} and \cite[Proposition 2]{Shlyafer}).
\end{itemize}
\end{lemma}

We precise \cite[Corollary 2.10]{Ster} as the following.

\begin{proposition} Let $R$  be a ring with the additive $p$-group $R^+$. The set  $N(R)$ is a subring of $R$ if and only if $N(R)^\circ$ is a normal Sylow $p$-subgroup of $R^\circ$.
\end{proposition}
\begin{proof}  Clearly,  $pR$ is an ideal of $R$ such that  $pR\subseteq N(R)$. The groups $(R/pR)^\circ$ and $R^\circ /(pR)^\circ $ are isomorphic, so we can assume that $pR=0$.

$(\Rightarrow )$ Suppose that $N(R)$ is a subring. Then $N(R)^\circ$ is a $p$-subgroup  of $R^\circ$ in view of \cite[Lemma 2.4]{Amberg_Dickenschied} and   $N(R)^\circ$ is contained in some maximal (Sylow)  $p$-subgroup $S$ of $R^\circ$. If $g\in S\setminus N(R)^\circ$, then
\[
0=g^{(p^n)}=\mu_{p^n}(g)=g^{p^n}, \qquad\qquad (\text{for some}\quad  n\in \mathbb{N})
\] 
 so  $g\in N(R)$, a contradiction. Hence $N(R)^\circ$ is a Sylow $p$-subgroup of $R^\circ$.
If $S_1$ is a maximal (Sylow) $p$-subgroup  of $R^\circ$ and $h\in S_1$ has order $p^m$, then $0=h^{(p^m)}=h^{p^m}$, so  $h\in N(R)$ and $S_1=S$. Consequently,  $S$ is normal in   $R^\circ$.

$(\Leftarrow )$ Since $N(R)$ is closed under the circle operation "$\circ$"{}, $N(R)$ is a subring of $R$ by \cite[Theorem 2.1]{Ster}.
\end{proof}

Let $Dz(R)$ be the set of all left and right zero divisors and $0\in Dz(R)$.

\begin{proof}[Proof of Proposition \ref{XY0}]
$(i)$ Let  $R$ be not commutative and let $g\in N(R)$. If $g^2=0$, then $$0=g^{(n)}=\mu_n(g)=ng$$ for some $n\in\mathbb{N}$ and so  $g=0$. This means that $N(R)=0$. It follows that if  $ab=0$ for some nonzero $a,b\in R$, then  $(ba)^2=0$ what implies that $ba=0$.   Therefore  $Dz(R)$ is commutative and, consequently, $Dz(R)$ is an ideal of $R$ by \cite[Theorem 5.5]{Klein_Bell_07}.
Moreover, $Dz(R)^2C(R)=0$ \cite[Lemma 5.4]{Klein_Bell_07},  so $C(R)^3=0$. Thus  $C(R)=0$, which is  a contradiction.

$(ii)$ Indeed,  $J(R)=N(R)=0$ by Lemma \ref{krempa}.
If $ac=0$ for some $a,c\in R$, then  $(ca)^2=c(ac)a=0$ implies that $ca=0$. As a consequence, $0=a(cR)=cRa$ and $c=0$ or $a=0$ by the primeness of $R$. Hence $R$ is a domain. The rest follows from \cite[Proposition 4]{Shlyafer}.
\end{proof}

\begin{proposition} \label{XY1} Let $R$ be an unitary domain of characteristic $p>0$. If $U(R)$ is torsion, then it is a  $p'$-group and the following statements hold:
\begin{itemize}
\item[$(i)$]  $R^\circ \setminus \{ 0\}\subseteq U(R)$ and $I\cap R^\circ =0$ for any proper right (left) ideal $I$ of $R$; in particular, $J(R)=0$;
\item[$(ii)$] if $S$ is a subring of $R$ and $S\cap R^\circ \neq 0$, then $1\in S$.
\end{itemize}
\end{proposition}
\begin{proof}
Obviously,  $N(R)=0$. Since $g^{(p)}=\mu_p(g)=g^p$ for    $g\in R^\circ$, we deduce that $g^{(p)}\neq 0$ and  $R^\circ$ is a $p'$-group.

$(i)$ Let $I$ be a proper right (left) ideal of  $R$. If $0\neq b\in I\cap R^\circ$, then
\[
\textstyle0=b^{(n)}=\mu_n(b)=b\big(n+\sum\limits_{k=2}^n\binom{n}{k}b^{k-1}\big)
\]
for some $n\in \mathbb N$. If  the  great common divisor $\text{GCD}(n,p)=1$, then  $n\in I$ and  there exist $u,v\in\mathbb{Z}$ such that  $r=(nr)u+(pr)v\in I$ for any $r\in R$. Consequently,   $I=R$, a contradiction. This implies that $I\cap R^\circ =0$.

Inasmuch as $x\in (Rx)^\circ \cap R^\circ$ ($x\in (xR)^\circ \cap R^\circ$, respectively) for each $x\in R^\circ$,  we conclude that $xR=R=Rx$, so  each  nonzero quasi-invertible element is invertible in $R$.

 $(ii)$  If  $S$ is a nonzero subring of $R$ and $0\neq b\in  S\cap R^\circ$, then, as above, $n\in S$ and consequently $1\in S$.\end{proof}

According to Proposition \ref{XY1}, we can ask the  following questions:
\begin{itemize}
\item[Q1.] Does there exist a  unitary infinite  non-commutative simple ring $R$ of characteristic $p>0$ with the
torsion unit group $U(R)$?
\item[Q2.] Does there exist a  unitary (infinite)  non-commutative  ring $R$ which is not a skew field, such that  $R^\circ \setminus \{ 0\}\subseteq U(R)$?
\end{itemize}

%=======================================================
%\section*
\section{Locally finite rings}

%================================================

{We start with  some properties of locally finite rings.}

\begin{lemma} \label{LF1} If $R$ is a  locally finite ring with unity, then the following statements hold:
\begin{itemize}
\item[$(i)$] $R^+$ is a torsion  $\pi$-group for some set
$\pi$ of primes;
\item[$(ii)$] $U(R)$ is locally finite;
\item[$(iii)$] $1+J(R)$ is a  locally nilpotent $\pi$-group.
\end{itemize}
\end{lemma}

\begin{proof} $(i)$ Obviously.

$(ii)$ If $X$ is a finite subset of $U(R)$, then $\langle X\rangle
\subseteq {\langle X\rangle}_{\rm rg}$ and so the subgroup
$\langle X\rangle$ is finite.

$(iii)$      The unipotent subgroup $1+J(R)$ of $U(R)$ is locally nilpotent \cite[Corollary 1]{Amberg_Dickenschied_Sysak}. Since  $J(R)$ is nil,  $1+J(R)$ is a $\pi$-group in view of \cite[Lemma 2.4]{Amberg_Dickenschied}. \end{proof}

\begin{lemma} \label{SYS} If $P$ is a  minimal prime ideal of a $2$-torsion-free ring $R$, then $\charak R/P\neq 2$.
\end{lemma}
\begin{proof} Let $X=\{ 2^na \mid a\in R\setminus P \text{  and  }  n\in\mathbb{N}\cup \{0\}\}$. Clearly,  $X$ is non-empty, $0\notin X$ and $X$ is an $m$-system (in the sense of \cite{McCoy}). Therefore, there exists  a two-sided ideal $M$ of $R$ which is maximal to being disjoint from $X$ (then $M$ is prime by \cite[Lemma 4]{McCoy} and $M\subseteq R\setminus X$). Since $R\setminus P\subseteq X$, we conclude that $M\subseteq P$ and consequently $M=X$. Hence $\charak R/P\neq 2$. \end{proof}

If for each $x\in R$ there exists $n\in\mathbb{N}$ with $x^n=x$, then $R$ is commutative by a well-known theorem of N.~Jacobson.   A ring $R$ is called {\it
periodic} if, for each $x\in R$, there exist different positive integers $m$ and $n$,  such that  $x^m=x^n$.

\begin{lemma} \label{LF2} Let $R$ be a   locally finite ring with unity. The following statements hold:
\begin{itemize}
\item[$(i)$]  $R$ is periodic;
\item[$(ii)$]  if $R$ is   $2$-torsion-free semiprime, then it is commutative;
\item[$(iii)$] if $R$ is  prime of $\charak R\not=2$, then it is a field;
\item[$(iv)$] if $R$ is   $2$-torsion-free, then $C(R)\subseteq {\mathbb P}(R)=N(R)=J(R)$ (i.e., $R$ is $2$-primal and right quasi-duo).
\end{itemize}
\end{lemma}

\begin{proof}
$(i)$ For the proof, see \cite[Corollary 2]{Hirano91}. %Inasmuch as $U(R)$ is a locally finite group, the result follows in view of Lemma %\ref{LF0}.

Let $R$ be a $2$-torsion-free ring.

$(ii)$ It holds in view \cite[Thereom 4.5]{Bell_Yaqub} and the part $(i)$.

$(iii)$ It follows from  the part $(i)$ and the fact that any periodic domain is a field.

$(iv)$ It is a consequence of  parts $(ii)-(iii)$ and Lemma \ref{SYS}. \end{proof}

\begin{proof}[Proof of Proposition \ref{LF3}.]
$(i)$ The quotient ring $R/P$ is a field for each  prime ideal $P$ of $R$ by \cite[Corollary 2.6]{Huh_Kim_Lee04} and Lemma \ref{LF2}$(ii)$. Thus the part $(i)$ holds.

$(ii)$ Since $R$ is strongly $\pi$-regular and $J(R)\subseteq N(R)$ by \cite[Theorem 1]{Hirano78}, we conclude that $J(R)=N(R)$ in view of Lemma  \ref{LF2}$(iii)$. Hence $R$ is exchange of stable range $1$ by \cite[Theorem 5.23 and Proposition 5.6]{Tuganbaev02}. Moreover, $R$ is abelian by \cite[3.20(3)]{Tuganbaev02}, $R/J(R)$ is a subdirect product of fields and so each element of  $R$ is a sum of an invertible and a central elements by  \cite[Thereom 6.29]{Tuganbaev02}. \end{proof}

\begin{corollary} A locally finite $2$-torsion-free ring $R$ is right (left) Ore, i.e., there exists the classical right (left)  quotient ring $Q(R)$.
\end{corollary}
\begin{proof} The assertion holds in view of Lemma \ref{LF2}$(iii)$, Proposition \ref{LF3}$(ii)$ and \cite[Theorem 2.1 and Proposition 1.9(5)]{Kim_Kwak_Lee}. \end{proof}

Since each  $2$-torsion-free locally finite ring is abelian $\pi$-regular, we provide the following.

\begin{corollary} An abelian $\pi$-regular ring $R$ satisfies the K\" othe's conjecture, i.e., the sum of two nil left ideals is always nil.
\end{corollary}
\begin{proof} The set $N(R)$ of nilpotent elements is an ideal of $R$ by \cite[Theorem 2]{Badawi}. The rest follows from \cite[Lemma 1.4(2) and Theorem  2.1(2)]{Kim_Kwak_Lee}. \end{proof}

\begin{lemma} \label{SS24} \cite[Lemma 18.34B]{Faith2}
Let $R$ be a right Noetherian ring. If $R/P$ is an Artinian ring for each prime ideal $P$ of $R$, then $R$ is  a prime ring or $R$ is a right  Artinian ring.
\end{lemma}

\begin{proposition} \label{PiPi1}
A ring $R$ is locally finite right Noetherian  if and only if it
is  a locally finite right Artinian.
\end{proposition}
\begin{proof}
$(\Leftarrow )$ Each right Artinian ring is right
Noetherian by \cite[Theorem 18.13]{Faith2}.

$(\Rightarrow )$ Since $R/P$ is a field for each prime ideal $P$
of $R$ (see Lemma \ref{LF2}$(iii)$), we deduce that $R$ is right Artinian in
view of  Lemmas \ref{SS24} and  \ref{LF2}$(iii)$.
\end{proof}

%==========================

%\section{Proofs}

%=========================

\begin{proposition} \label{FSF}
Let  $R$ be  a  semilocal ring. The following conditions are equivalent:
\begin{itemize}
\item[$(i)$] $R$ is a locally finite ring;
\item[$(ii)$] the unit group $U(R)$ is locally finite;
\item[$(iii)$] $R^+$ is a torsion group, $J(R)$ is a locally nilpotent ideal and $R/J(R)={\sum\limits_{i=1}^n}^\oplus M_{m_i}(D_i)$ is a finite  direct sum of  rings of $m_i\times m_i$ matrices over locally finite fields $D_i$ with   $i=1,\ldots, n$.
\end{itemize}
\end{proposition}
\begin{proof}
%[Proof of Proposition $\ref{FSF}$]
$(i) \Rightarrow (ii)$ It follows from Lemma \ref{LF1}$(ii)$.

$(ii) \Rightarrow (i)$ It is clear that $R/J(R)$ is a finite direct
ring sum and each  direct summand is a  locally finite field or a
finite total matrix ring. The unipotent group $1+J(R)$ is locally
nilpotent group.  That yields  the subring ${\langle J(R)^{\circ}\rangle}_{\rm rg}=J(R)$ is locally nilpotent by
\cite[Lemma 3]{Amberg_Dickenschied_Sysak}.

Let $X$ be a finite subset of $R$. There exists an additive
group isomorphism
$$
({\langle X\rangle}_{\rm rg}+J(R) )/J(R)\cong
{\langle X\rangle}_{\rm rg}/({\langle X\rangle}_{\rm
rg}\bigcap J(R))
$$ and
$$({\langle X\rangle}_{\rm rg}+J(R) )/J(R)={\langle
{X}+J(R)\rangle }_{\rm rg}
$$ is a finite subring of $R/J(R)$. The subring $B:={\langle X\rangle}_{\rm
rg} \cap J(R)$ is finitely generated by \cite[Theorem 2]{Lewin}. Clearly,  it is
nilpotent and so $B/B^2$ is a finitely generated  ${\mathbb
Z}_n$-module for some $n\in\mathbb{N}$.  It implies that  the  subring $B^2$ is
finitely generated by \cite[Theorem 2]{Lewin}. Using induction on the
nilpotency index of $B$, we obtain that $B$ (and consequently
${\langle X \rangle }_{\rm rg}$) is finite.

$(ii) \Rightarrow (iii)$ For each $D_i$   $(i=1,\ldots, n)$  there exists a chain $F_1\subseteq F_2\subseteq \cdots$ such that $D_i=\cup_j F_j$ and  each $F_j$ is a finite subfield of the field $F_{j+1}$. Thus  $(R/J(R))^\circ$ is locally finite. Moreover,  the adjoint group $J(R)^\circ$ is locally finite, $(R/J(R))^\circ \cong R^\circ /J(R)^\circ$ and so $R^\circ$ is locally finite.

$(iii) \Rightarrow (ii)$ It is obvious.
\end{proof}

\begin{corollary}
A locally finite semilocal ring is semiperfect.
\end{corollary}

%==========================

\section{Properties induced by derivations}

%=========================

\begin{proposition} Let $R$ be a commutative ring with unity. If $R$
has a derivation $\delta$ with the finite kernel $\Ker \delta$,
then $R$ is a  locally finite ring.  The prime radical ${\mathbb P}(R)$ has
finite index in $R$ and $\delta  (R)\subseteq  {\mathbb P}(R)$.
\end{proposition}
\begin{proof}  Assume that $R$ is infinite. Obviously,  $\delta (1)=0$ and  $\Ker \delta \not=0$.
This implies that $nR=0$ for  some $n\in\mathbb{N}$ and  $R$ is a finite direct ring  sum of $p$-components $F_p$,
where the prime $p$ divides $n$ and
\[
F_p=\{ r\in R\mid p^kr=0\quad  \text{for some}\quad  k=k(r)\in\mathbb{N}\bigcup\{0\}\}.
\]
Consequently, without loss of generality,   we can assume that
$n=p^s$ for some  $s\in {\mathbb N}$. Since $$\varphi: R/pR\ni
a+pR\mapsto p^ka+p^{k+1}R\in p^kR/p^{k+1}R$$ is an additive group
isomorphism,  $p^kR/p^{k+1}R$ is infinite for any
$k=0,\ldots, s-1$.
If $\delta (R)\subseteq pR$, then $\delta (p^{s-1}R)=0$ and so
$p^{s-1}R\subseteq \Ker \delta$, a contradiction. Hence $\delta
(R)\nsubseteq pR$. Inasmuch as  $\delta (pR)\subseteq pR$, the rule
$$\Delta: R/pR\ni a+pR\mapsto \delta (a)+pR\in R/pR$$ determines a
nonzero derivation $\Delta$ of $R/pR$ and $\Ker \Delta $ is finite. Then
$\Delta (\beta^p)=0$ for any $\beta  \in R/pR$ and so the set $\{
\alpha^p\mid \alpha \in R/pR\}$ is finite. If $\alpha, \beta \in
R/pR$ are distinct elements and $\alpha^p-\beta^p=0$, then
$\alpha -\beta \in {\mathbb P}(R/pR)$. This implies that the index
$|R/pR:{\mathbb P}(R/pR)|<\infty$. However $pR\subseteq {\mathbb
P}(R)$ and so $|R: {\mathbb P}(R)|<\infty$.

Since ${\mathbb P}(R)$ is nil with the torsion additive group
${\mathbb P}(R)^+$, we conclude that the adjoint group ${\mathbb
P}(R)^{\circ}$ is locally finite. Thus $R$ is a semiperfect ring
with the torsion unit group $U(R)$, so  $R$ is locally finite by
Proposition \ref{FSF}. \end{proof}

\begin{proof}[Proof of Theorem $\ref{GGX}$.]
$(i)$ Assume that $R$ is commutative, $d\in \Delta$, $\delta \in Z(\Delta )$ and $a\in R$. Then
\begin{equation}
\label{A*}
0=[\delta, ad]=\delta (a)d+a[\delta, d]=\delta (a)d.
\end{equation}
The   ideal of $R$  generated by the set $\{ \mu (R)\mid \mu \in Z(\Delta )\}$ we denote by $\Delta_1(R)$. Then $\Delta_1(R)^2=0$ in view of $(\ref{A*})$. Since $d(\delta (R))=\delta (d(R))\subseteq \delta (R)$, we conclude that $\Delta_1(R)$ is a $\Delta$-ideal of $R$ and so \begin{equation} \label{1003*} \overline{d}:R/\Delta_1(R)\ni a+\Delta_1(R)\mapsto d(a)+\Delta_1(R)\in R/\Delta_1(R)
 \end{equation}
is a derivation of $B:=R/\Delta_1(R)$. Then $\overline{\Delta}=\{ \overline{d}\mid d\in \Delta \}$ is a subring of the Lie ring $\Der B$ and a left $B$-module. As before, $\overline{\Delta}_1(B)$ (and its inverse image in $R$) is  nilpotent, where $\overline{\Delta}_1:=\{ \overline{d}\mid d\in \Delta_2\}$ and $\Delta_2$ is an inverse image of $Z(\Delta /Z(\Delta ))$ in $\Delta$. Thus  $\Delta (R)$ is nilpotent according to the induction on nilpotent length of $\Delta$.

{   Now, assume that $R$ is not necessary commutative.} If $\delta \in \Delta$, then $\delta (C(R))\subseteq C(R)$ and the rule  $$ \overline{\delta}:R/C(R)\ni x+C(R)\mapsto \delta (x)+C(R)\in R/C(R)$$ determines a derivation of $R/C(R)$. Since $\overline{\Delta}=\{ \overline{\delta}\mid \delta \in \Delta \}$ is a left $(R/C(R))$-module, $\overline{\delta}(R/C(R))^m=\overline{0}$ for some $m\in\mathbb{N}$ and consequently $\Delta (R)^m\subseteq C(R)$.

$(ii)$ Suppose that $\gamma_{n+1}(\Delta )=0$ and $R$ is commutative. If $d \in \Delta^{(n-1)}$, $a\in R$, then
$$\begin{array}{rccl}d (a)d &= &[d, ad]& \in \gamma_2(\Delta ),\\
d^2 (a)d &= &[d, d(a)d]& \in \gamma_3(\Delta ),\\
{} & \vdots & {}& {}\\
d^{n-1} (a)d &= &[d,d^{n-2}(a)d]& \in \gamma_n(\Delta ),\\
d^{n} (a)d &= &[d,d^{n-1}(a)d]& \in \gamma_{n+1}(\Delta )=0.
 \end{array}$$
 This implies that $d^n(R)d=0$. Since $d(C(R))\subseteq C(R)$,
the result can be obtain  similar to that of the part $(i)$.

$(iii)$ Let $\delta \in \Delta$. If $R$ is commutative, then $a\delta \in \Delta$ for any $a\in R$ and so
\[
(-1)^n\delta^n(a)\delta=(-1)^n[a\delta, \underbrace{\delta,\ldots, \delta}_{n\quad \text{\tiny  times}}]=0
\] 
 for some $n\in\mathbb{N}$. Hence $\delta^n(a)\delta =0$.

{   Now, assume that $R$ is not necessary commutative.
Then  $\delta (C(R))\subseteq C(R)$  and, by the same argument as in the part $(i)$, there exist} $n=n(\delta, a)\in {\mathbb N}$ such that $\delta^n(a)\delta (R)\subseteq C(R)$ and the assertion holds.
\end{proof}

If $x\in R$, then the rule $\partial_x:R\ni a\mapsto (ax-xa)\in R$ determines a derivations $\partial_x$ of $R$; this derivation is called {\it an inner derivation} of $R$ ({\it induced by} $x$). The set $\IDer R$ of all inner derivations of $R$ is an ideal of the Lie ring $D$.

\begin{proposition}\label{PPI} Let $R$ be a  ring. The following statements hold:
\begin{itemize}
\item[$(i)$]  $R$ is Lie solvable (Lie nilpotent,  $n$-Engel,  Engel, locally Lie nilpotent, locally Lie solvable, respectively) if and only if the Lie ring $\IDer R$ is  solvable (nilpotent,  $n$-Engel, Engel, locally  nilpotent, locally  solvable, respectively);
\item[$(ii)$] if $D$ is  solvable (nilpotent,  $n$-Engel, Engel, locally  nilpotent, locally  solvable, respective\-ly), then $R$ is Lie solvable (Lie nilpotent,  $n$-Engel,  Engel, locally Lie nilpotent, locally Lie solvable, respectively);
\item[$(iii)$] if $R$ is $2$-torsion-free semiprime and  $D$ ($\IDer R$, respectively) is solvable, then $D=0$ ($R$ is commutative, respectively);
\item[$(iv)$] if $R$ is with unity of characteristic $0$ ($\charak (R/{\mathbb P}(R))=0$, respectively) and $D$  is solvable, then $C(R)\subseteq D(R)\subseteq {\mathbb P}(R)$;
\item[$(v)$] if $R$ consists from countable many elements and $D$ ($\IDer R$, respectively) is Lie solvable, then $D=0$ ($R$ is commutative, respectively);
 \item[$(vi)$] if $R$ is commutative and $D$ is locally nilpotent, then, for each $a\in R$ and  $\delta \in D$ there exists  $n=n(a,\delta)\in {\mathbb N}$ such that $\delta^n(a)\delta =0$. Moreover, if $R$ is reduced, then $\delta^n(a)=0$;
 \item[$(vii)$] if $R$ is a semiprime ring with the $n$-Engel derivation ring $D$ ($\IDer R$, respectively), then $R$ is commutative and, for each  $\delta \in D$ ($\delta \in \IDer R$, respectively), there exists  $n=n(\delta )\in\mathbb{N}$ such that $\delta^n(R)=0$. Moreover, if $F(R)=0$, then $D=0$.
\end{itemize}
\end{proposition}
\begin{proof} $(i)$ Since $[\partial_{x_1},\partial_{x_2},\cdots, \partial_{x_n}]=\partial_{[x_1,x_2,\cdots, x_n]}$ for any $x_1,x_2,\ldots, x_n\in R$ and $$\IDer \ni \partial_x\mapsto x+Z(R)\in R^L/Z(R)$$ is a Lie ring isomorphism, the result is obvious.

$(ii)$ It is immediately.

$(iii)$ {  If $A\subseteq R$, then by $\Delta_A$ we denote the set $\{ \partial_a\mid a\in A\}$. If $A$ is a Lie ideal of $R$, then $\Delta_A$ is an ideal of $\IDer R$.
} Assume that $D$ is solvable of length $n>1$. If $D^{(n-1)}\cap \IDer R\neq 0$, then
\[
I_{D^{(n-1)}}:=\{ a\in R\mid  \partial_a\in D^{(n-1)}\bigcap \IDer R\}
\]
is a Lie ideal of $R$ and
\[
0=[D^{(n-1)},D^{(n-1)}]\supseteq [\Delta_{I_{D^{(n-1)}}}, \Delta_{I_{D^{(n-1)}}}]=\Delta_{[I_{D^{(n-1)}},I_{D^{(n-1)}}]}
\]
what gives that
$[I_{D^{(n-1)}},I_{D^{(n-1)}}]\subseteq Z(R)$. Thus    $I_{D^{(n-1)}}\subseteq Z(R)$ by \cite[Lemma 1]{Herstein70} and consequently $\Delta_{I_{D^{(n-1)}}}=0$, a contradiction. Hence $D^{(n-1)}\cap \IDer R=0$. This implies that $R$ is commutative,  $D(R)^m=0$ for some $m\in\mathbb{N}$ by Theorem \ref{GGX}$(i)$ and  $D(R)=0$ by the semiprimeness of $R$. We conclude the result.

$(iv)$ Each  minimal prime ideal of $R$ is closed with respect to each $\delta \in D$  by \cite[Proposition 1.3]{Goodearl_Warfield}. The  map
\begin{equation} \label{PPI1}
 \overline{\delta}:R/P\ni a+P\mapsto \delta (a)+P\in R/P
\end{equation}
is a derivation of a prime ring $R/P$ of characteristic ${}\neq 2$ by Lemma \ref{SYS} and $\Der (R/P)=0$ by the part $(iii)$. Thus $C(R)\subseteq D(R)\subseteq {\mathbb P}(R)$ ($C(R)\subseteq {\mathbb P}(R)$, respectively).

$(v)$ There exists a collection of prime ideals $P_{\beta}$ $(\beta \in \Gamma )$ by \cite[Theorem 2.1]{Chuang_Lee} such that
\[
\bigcap_{\beta \in \Gamma}P_{\beta}=0 \quad \text{and}\quad \delta (P_{\beta})\subseteq P_{\beta}\qquad (\forall \delta \in D \text{ and/or } \forall \delta \in \IDer R, \text{  respectively}).
\]
Consequently,   $\overline{\delta}$ defined by the rule $(\ref{PPI1})$ is a derivation of $R/P$ for any $P=P_\beta$ and  $\Der (R/P)=0$  by the part $(iii)$. Hence
\[
C(R)\subseteq D(R)\subseteq \bigcap_{\beta \in \Gamma}P_{\beta}=0
\]
 and the assertion holds.

 $(vi)$ The subring  of $D$ generated by derivations $d$ and $ad$\ $(a\in R\ \text{and}\ d\in D)$ is nilpotent and the result holds by the same argument as in the proof of Theorem $\ref{GGX}(ii)$.

$(vii)$ Since $R$ is $n$-Engel it is commutative by Corollary \ref{PP}. If $\delta \in D$ ($\delta \in \IDer R$, respectively), then by the same argument, as in the proof of Theorem \ref{GGX}$(ii)$, we obtain that $\delta^n(R)=0$ for some $n\in\mathbb{N}$.

Let $F(R)=0$.   We prove that $\delta =0$ using induction by $n$. If $\delta^2(R)=0$, then $\delta =0$ by \cite[Corollary 1]{Carini}. Let $n>2$ and suppose that $\delta^{n-1}(R)=0$ implies that $\delta =0$. Assuming $\delta^n(R)=0$ we see that
\[
0=\delta^{n}(a\delta^{n-2}(b))=(n-1)\delta^{n-1}(a)\delta^{n-1}(b)\qquad (\forall a,b\in R)
\]
what implies that $(\delta^{n-1}(R))^2=0$ and hence $\delta^{n-1}(R)=0$  by the semiprimeness of $R$. Thus  $\delta =0$.
\end{proof}

Recall that  the commutator ideal $C(R)$ of a  $2$-torsion-free  Lie solvable ring $R$ is nil (see \cite[Theorem 2.1]{Sharma_Srivastava85} and \cite[Theorem]{Zalesskii_Smirnov_82}). Proposition \ref{PPI}$(v)$ precise this result in the countable case.

\begin{corollary}\label{hoho}
Let $R$ be an algebra  over a field ${\mathbb F}$ of characteristic $0$. If the  Lie ${\mathbb F}$-algebra  $D$ ($\IDer R$, respectively) is Engel, then $R$ is Lie nilpotent.
\end{corollary}
\begin{proof} Since $\IDer R$ is an Engel Lie algebra over $Z(R)$ and $\IDer R$ and $R/Z(R)$ are isomorphic as Lie algebras, $R$ is Lie Engel and so  it is nilpotent by \cite[Theorem B]{Riley98}.
\end{proof}

%=======================================================
%\section*
\section{Solvability}

%================================================

\begin{lemma} \label{LSLN}
Each  Lie solvable nil ring $R$ is locally nilpotent. Moreover, if $R$ is a ${\mathbb Q}$-algebra, then it is Lie nilpotent.
\end{lemma}
\begin{proof} The ring $R$ contains a nilpotent ideal $I$ such that $R/I$ satisfies the identity $$[x_1,[x_2,x_3],[x_4,x_5]]=0$$
(see \cite[Theorem 2.1]{Sharma_Srivastava85} and \cite[Theorem]{Zalesskii_Smirnov_82}). This implies that every finitely generated subring of $R/I$ is nilpotent  (see e.g. \cite[Theorems 6.3.3 and 6.3.39]{Rowen}), so  $R$ is a locally nilpotent ring.
If $R$ is a ${\mathbb Q}$-algebra, then $R$ is locally nilpotent as an algebra and, by \cite[Theorem B]{Riley98}, it is Lie nilpotent.
\end{proof}

\begin{proof}[Proof of Theorem $\ref{ZAZ}$.]
$(i)$  The quotient ring $R/{\mathbb P}(R)$ is semiprime and so its adjoint group $(R/{\mathbb P}(R))^\circ$ is abelian by Corollary \ref{PP}. Moreover,   $(R/{\mathbb P}(R))^\circ \cong R^\circ /{\mathbb P}(R)^\circ $ and ${\mathbb P}(R)^\circ$ is nilpotent by \cite{Jennings47}.

$(ii)$ {  Let $P$ be a prime ideal of $R$. If $J(R/P)$ is nonzero, then $[J(R),R]\subseteq P$ in view of \cite[Theorem A]{Catino_Miccoli_Sysak}. This implies that $[J(R),R]\subseteq {\mathbb P}(R)$.

Since $R/{\mathbb P}(R)$ is semiprime, then $J(R)\subseteq {\mathbb P}(R)$ or $J(R/{\mathbb P}(R))$ is commutative in view of \cite[Theorem B]{Catino_Miccoli_Sysak}. From this it follows that $J(R/{\mathbb P}(R))^2\cdot C(R/{\mathbb P}(R))=0$ and consequently
$J(R/{\mathbb P}(R))\cdot C(R/{\mathbb P}(R))=0$ what gives that $$J(R)\cdot C(R)\subseteq {\mathbb P}(R).$$}

 $(iii)$ Let $R$ be a $2$-torsion-free Lie solvable ring.

$(a)$  First, assume that $R$ is prime of solvable length $n>1$. Since $[R^{(n-1)},R^{(n-1)}]=0$, we conclude that $R^{(n-1)}\subseteq Z(R)$ by \cite[Lemma 1]{Herstein70}.  But then $[R^{(n-2)},R^{(n-2)}]\subseteq R^{(n-1)}$ and we obtain a contradiction in view of \cite[Lemma 1]{Herstein70}. Hence $R$ is commutative. This implies that $C(R)\subseteq {\mathbb P}(R)=N(R)$ in general case. If $R$ has unity,  then $R$ is abelian in view of \cite[3.20]{Tuganbaev02} and
$N(R)$ is a locally nilpotent ring by Lemma \ref{LSLN}.

$(b)$  If $R^\circ$ is torsion, then it is locally finite  (and so it locally nilpotent by \cite[Corollary 2]{Amberg_Dickenschied_Sysak}).

$(c)$ If $N(R)^+$ is torsion-free divisible, then $N(R)$ is a locally nilpotent $\mathbb Q$-algebra and so it satisfies the Engel condition. The algebra   $N(R)$ is Lie nilpotent  by \cite[Theorem B]{Riley98} and $N(R)^\circ$ is nilpotent  by \cite{Jennings47}, as required.
\end{proof}

\begin{corollary} Let $R$ be a right Goldie ring (or $R$ satisfies the ascending chain condition on both left and right annihilators) with unity.
 If $R$  satisfies one of the conditions:
\begin{itemize}
\item[$(i)$] $R$ is Engel   as a Lie ring;
\item[$(ii)$] $R$ is $2$-torsion-free locally finite,
 \end{itemize}
 then $R$ is Lie solvable and the unit group $U(R)$ is nilpotent-by-abelian.

\end{corollary}
\begin{proof}   We have that  $C(R)\subseteq N(R)$ by Corollary \ref{19} (by Lemma \ref{LF2}, respectively), and so $N(R)$ is an ideal of $R$. Since $N(R)$ is nilpotent by \cite[Theorem 1]{Lanski69} (by \cite[Theorem 1]{Herstein_Small}, respectively) and $R/N(R)$ is commutative, $R$ is Lie solvable and the unit group $U(R)$ is nilpotent-by-abelian.
\end{proof}

\begin{remark} If $R$ is a right Goldie $n$-Engel ring of prime characteristic $p>0$ and $n<p$, then it is Lie nilpotent in view of  \cite{Higgins54}.
\end{remark}

%==========================

\section{Local  rings}

%=========================

\begin{proof}[Proof of Theorem $\ref{LSS}$.] $(i)$ The group unit $U(R/J(R))$ is solvable what gives that $[R,R]\subseteq J(R)$ by \cite[Theorem 2]{Zalesskii65}. Since $J(R)^\circ$ is solvable, we deduce that  $J(R)$ (and consequently $R$) is Lie solvable by \cite[Theorem A]{Amberg_Sysak_III}. Moreover, the Levitzki radical $L(R)$ of $R$ is a $PI$-ring    by \cite[Theorem B(2)]{Amberg_Sysak_III} and so it is locally nilpotent. If $B=R/L(R)$, then $J(B)$ is commutative in view of \cite[Theorem B]{Amberg_Sysak_III} what implies that $C(B)^3=\overline{0}$ and consequently $B$ is commutative. Hence $C(R)\subseteq L=N(R)$.

$(ii)$ Since     $N(R)$  is a locally nilpotent ideal of $R$, the $\mathbb Q$-algebra   $N(R)$ is Lie nilpotent by \cite[Theorem B]{Riley98} and so  the adjoint group $N(R)^\circ$ is nilpotent by \cite{Jennings47}.\end{proof}

\begin{proposition} \label{XY2} Let $R$ be a local ring. The following statements are equivalent:
\begin{itemize}
\item[$(i)$] the unit group $U(R)$ is torsion;
\item[$(ii)$] $J(R)$ is nil and $R/J(R)$ is an absolute field of characteristic $p>0$;
\item[$(iii)$] $U(R)\cong (1+J(R))\rtimes U(R/J(R))$, where $1+J(R)$ is a $p$-group and $U(R/J(R))$ is a $p'$-group.
\end{itemize}
Therefore, a local ring $R$ with the torsion unit group $U(R)$ is a locally finite ring.
\end{proposition}
\begin{proof} $(i)\Rightarrow (ii)$ Since $R/J(R)$ is a skew field and the unit group $U(R/J(R))$ is torsion, we deduce that $R/J(R)$ is commutative and  $p(R/J(R))=0$ for some prime $p$. Hence $pR\subseteq J(R)$  and $p^kR=0$ for some  $k\in {\mathbb N}$ in view of Lemma $\ref{krempa}(i)$. Since $U(R/J(R))$ is torsion, $R/J(R)$ is an absolute field.

$(ii)\Rightarrow (iii)$ and $(iii)\Rightarrow (i)$ are obviously.
\end{proof}

\begin{lemma} \label{C101C} Let $R$ be a local ring  which is Engel as a Lie ring. If $F(R)=0$, then either $R$ is a Lie nilpotent $\mathbb Q$-algebra or $pR\subseteq J(R)$ for some prime $p$, $C(R)\subseteq N_r(R)=N(R)$ and $N(R)^\circ$ is a torsion-free group. Moreover, if in the last case, $R$ is $n$-Engel, then it is Lie solvable and $N(R)^\circ$ is nilpotent (the $R^\circ$ is nilpotent-by-abelian and $R$ is Lie solvable).
\end{lemma}
\begin{proof} If $R^+$ is a divisible group, then $R$ is a $\mathbb Q$-algebra and so it is Lie nilpotent by \cite[Theorem B]{Riley98}. Therefore, we assume in the next that $pR$ is proper in $R$ for some prime $p$. Then $pR\subseteq J(R)$ and   $C(R)\subseteq N_r(R)= N(R)$ by Corollary \ref{19}. Consequently, $N(R)$ is an ideal of $R$ and  the adjoint group $N(R)^\circ$ is torsion-free by \cite[Lemma 2.4]{Amberg_Dickenschied}.

Assume that $R$ is $n$-Engel. Since $N(R)^\circ$ is $m$-Engel for some $m\in\mathbb{N}$ by \cite[Main Theorem]{Amberg_Sysak_I}, it is nilpotent by the theorem of Zelmanov \cite{Zelmanov89}. Then $N(R)$ is Lie nilpotent by \cite{Jennings47} and consequently $R$ is Lie solvable.
\end{proof}

\begin{proof}[Proof of Theorem $\ref{FLE1}$.] In view of Lemma \ref{C101C}, we assume that $pR\subseteq J(R)$ for some prime $p$. Since $C(R)\subseteq N(R)\subseteq J(R)$ and $(J(R)/C(R))^\circ \cong J(R)^\circ /C(R)^\circ$ is abelian, we deduce that $N(R)$ is an ideal of $R$ and $J(R)^\circ$ is a solvable group. Moreover, $J(R)^\circ$ is $m$-Engel group for some  $m\in\mathbb{N}$ depending on $n$ by \cite[Main Theorem]{Amberg_Sysak_I}. Then the adjoint group $J(R)^\circ$ is locally nilpotent by \cite[Theorem 1]{Gruenberg}.

If $0\neq a\in \tau (J(R)^\circ )$, then $0=a^{(n)}=a(n+\sum\limits_{k=2}^n\binom{n}{k}a^{k-1})$ for some $n\in\mathbb{N}$. Hence $n\in J(R)$. Obviously, the order of each  element  of $U(R/J(R))$ is  relatively prime with $p$ and, thus,  $1\in J(R)$, a contradiction. Hence $\tau (J(R)^\circ )=0$ and, by theorem of Zelmanov  \cite{Zelmanov89}, $J(R)^\circ$ is nilpotent as a locally nilpotent  $m$-Engel torsion-free group. Since  $pR\subseteq J(R)$,  $\gamma_{n+1}(pR)=0$ for some integer $n\geq 0$ and
$$\begin{array}{rclrlrl}
\gamma_1(pR)&=&pR&{}&{}&=&p\gamma_1(R), \\
\gamma_2(pR)&=&[pR,pR]&=&p^2[R,R]&=&p^2\gamma_2(R),\\
&\vdots & {}&\vdots &{}&\vdots &{}\\
\gamma_{n+1}(pR)&=&[\gamma_n(pR),\gamma_n(pR)]&=&p^{n+1}[\gamma_n(R),\gamma_n(R)]&=&p^{n+1}\gamma_{n+1}(R).\end{array}$$
Thus we conclude that $\gamma_{n+1}(R)=0$, i.e., $R$ is Lie nilpotent.
\end{proof}

\begin{lemma} \label{XAL3} Let $R$ be a local ring with the nil Jacobson radical $J(R)$. If $R$ is Engel  and $\charak R/J(R)=0$, then $R$ is Lie nilpotent.
\end{lemma}
\begin{proof}   If $0\neq a\in R$ and $pa=0$ for some prime $p$, then $a\cdot pR=0$ and therefore $pR\subseteq J(R)$, a contradiction. Hence $F(R)=0$. If $qR\neq R$ for some prime $q$, then $qR\subseteq J(R)$ and, for any $x\in R$, there exists  $k=k(x)\in\mathbb{N}$ such that $q^kx^k=0$, a contradiction. Hence $R^+$ is a divisible group. As a consequence, $R$ is an algebra over the rational numbers field $\mathbb Q$ and  $R$ is Lie nilpotent by \cite[Theorem B]{Riley98}.
\end{proof}

%=======================================================
%\section*
\section{Corollaries}

%================================================

There are  large number of articles  which  extend  the  Cohen's Theorem \cite{Cohen} and Kaplansky Theorem \cite[Theorem 12.3]{Kaplansky49} (see e.g. \cite{Reyes, Szigeti_Wyk} and others). We also  present the following generalizations of these theorems.

\begin{corollary} \label{S25} Let $R$ be a ring with unity. If the commutator ideal $C(R)$ is nil (in particular, $R$ is  Engel), then the following statements hold:
\begin{itemize}  \item[$(i)$] if prime ideals of $R$ are finitely generated as right ideals, then the quotient ring $R/C(R)$ is a commutative Noetherian ring;
\item[$(ii)$] if $R$ is a right Noetherian ring and each its  maximal right ideal is principal, then $R/C(R)$ is a commutative principal ideal ring;
\item[$(iii)$] if each  prime ideal of $R$ is principal as a right ideal, then $R/C(R)$ is a commutative principal ideal ring;
\item[$(iv)$] if $R$ is  right Noetherian  and $R/{\mathbb P}(R)$ is finite, then $R/C(R)$ is finite.
\end{itemize}
\end{corollary}
\begin{proof} $(i)$--$(iii)$ Since $I+C(R)$ is an ideal of $R$ for its any right ideal $I$,  the result follows from Cohen's Theorem  \cite{Cohen} and  Kaplansky Theorem \cite[Theorem 12.3]{Kaplansky49}.

$(iv)$ Without loss of generality, assume that ${\mathbb P}(R)\neq 0$. Obviously,   $C(R)$ is nilpotent by \cite[Theorem 1]{Lanski69} and Corollary \ref{PP} and   $C(R)\subseteq {\mathbb P}(R)\subseteq P$ for any nonzero prime ideal $P$ of $R$. Finally,  $R$ is Artinian  by Lemma \ref{SS24} and so $R/C(R)$ is finite by \cite[Lemma 22]{Artemovych}.
\end{proof}

\begin{remark} If $R$ is a Lie nilpotent algebra over a field of characteristic $0$ and $I=I^2$ is a  f.g.  ideal (as a one-sided ideal) of $R$, then $I=eR$ for some central idempotent $e\in R$.
\end{remark}
In fact, this follows from \cite[Theorem 1]{Szigeti98} and Corollary \ref{PP}$(ii)$.

\begin{proposition} Let $R$ be a nil ring. The following statements  hold:
 \begin{itemize}
\item[$(i)$]  if  $R$ is $n$-Engel as a Lie ring and $F(R)=0$, then $R^\circ$ is nilpotent (and so $R$ is Lie nilpotent);
\item[$(ii)$]  if $R$ is $n$-Engel of bounded index, then $R/F(R)$ is Lie nilpotent (and so the adjoint group $R^\circ$ is locally nilpotent and  torsion-by-(torsion-free nilpotent)).
 \end{itemize}
\end{proposition}
\begin{proof} $(i)$ The adjoint group   $R^\circ$ is $m$-Engel for some $m\in\mathbb{N}$ by \cite[Main Theorem]{Amberg_Sysak_I} and, therefore, it is locally nilpotent by \cite[Lemma 2.2]{Amberg_Sysak_I}. As a consequence, $R^\circ$ is nilpotent by \cite{Zelmanov89} (see e.g. \cite[Corollary 1]{Amberg_Dickenschied_Sysak})
and   $R$ is Lie nilpotent by \cite{Jennings47}.

$(ii)$ In view of  the result of Levitzky \cite[Lemma 1.1]{Herstein69}, the ring $R$ is locally nilpotent and so $R^\circ$ is a locally nilpotent group.
If $F(R)=0$, then $R^\circ$ is nilpotent by \cite{Zelmanov89} and hence $R$ is Lie nilpotent by \cite{Jennings47}.
\end{proof}

\begin{proposition} Let $R$ be a ring such that $Dz(R)$ is  commutative. The following  holds:
\begin{itemize}
\item[$(i)$] either $Dz(R)^2\neq 0$ and $R$ is Lie metabelian (then the adjoint group $R^\circ $ is metabelian) or $Dz(R)^2=0$ (i.e., $R$ is $2$-primal); in the last case $R$ is commutative or $Dz(R)$ is completely prime ideal of $R$;
\item[$(ii)$] if $R$ is    Engel ($2$-torsion-free Lie solvable, $2$-torsion-free locally finite, respectively) with unity, then it is Lie  metabelian and $C(R)^3=0$.
\end{itemize}
\end{proposition}
\begin{proof} Assume that $R$ is non-commutative.

$(i)$ The set $N(R)$ is an ideal of $R$ by \cite[Theorem 5.7]{Klein_Bell_07}. If $Dz(R)^2\neq 0$, then $R$ is Lie metabelian by   \cite[Theorem 5.7]{Klein_Bell_07}.

Now, assume that    $Dz(R)^2=0$. Thus $Dz(R)=N(R)$ and $ab\in Dz(R)$ implies that $a\in Dz(R)$ or $b\in Dz(R)$ for each $a,b\in R$ what means that $Dz(R)$ is completely prime.

$(ii)$ If $Dz(R)^2=0$, then $N(R)^2=0$ and therefore $R$ is Lie metabelian in view of Corollary  \ref{19} (Theorem \ref{ZAZ}$(iii)$, Lemma \ref{LF2}$(iv)$, respectively).

From \cite[Theorem 5.7]{Klein_Bell_07} it follows that $R$ is Lie metabelian also in the case $Dz(R)^2\neq 0$.

If $D(R)^2\neq 0$, then $D(R)^2C(R)=0$ by \cite[Lemma 5.4]{Klein_Bell_07} and so $C(R)^3=0$.

 Finally, $Dz(R)^2C(R)=0$ by \cite[Lemma 5.4]{Klein_Bell_07} and so $C(R)^3=0$.
\end{proof}

\begin{remark} Let $R$ be a ring such that  the set $N(R)$ is commutative. If $R$ is an Engel ($2$-torsion-free  Lie  solvable, $2$-torsion-free locally finite, respectively), then $N(R)$ is an ideal of $R$, $C(R)\subseteq N(R)$ and $N(R)^2\cdot C(R)=0$.

\end{remark}

{  In fact,  $C(R)\subseteq N(R)$ by Corollary \ref{19} (by Theorem \ref{ZAZ}, Lemma \ref{LF2}$(iv)$, respectively)
what gives that $N(R)$ is an ideal of $R$ and the result follows.}

\begin{corollary}  Let $R$ be a ring with unity. The following statements hold:
\begin{itemize}
\item[$(i)$] if $R$ is right Noetherian $\pi$-regular and  satisfies  the $n$-Engel condition (is  $2$-torsion-free Lie solvable, respectively), then   it is a field or right Artinian;
\item[$(ii)$]  if $R$ is an Engel  ring of bounded index and $J(R)$ is nil, then 
\[
C(R)\subseteq {\mathbb P}(R)=N(R)=J(R),
\] 
$N(R)$ is a locally nilpotent ideal of $R$ and $U(R)$ is (locally nilpotent)-by-abelian.
\end{itemize}
\end{corollary}
\begin{proof}
$(i)$  The ring $R$ is abelian by Corollary \ref{PP} (by Theorem \ref{ZAZ}, respectively) and, therefore, $R/P$ is a field for any prime ideal $P$ by \cite[Theorem 5]{Badawi}. Hence each  prime ideal is a maximal right ideal. In view of Lemma \ref{SS24} and Corollary \ref{PP}$(i)$, $R$ is a field or right Artinian.

$(ii)$ {  We obtain that $C(R)\subseteq N(R)=J(R)$ by Corollary \ref{19}. Then $R/P$ is commutative for each prime ideal $P$ of $R$ and so $C(R)\subseteq {\mathbb P}(R)$. The ideal $N(R)$ is locally nilpotent in view of \cite[Lemma 1.1]{Herstein69} and so  the unit group $U(R)$ is (locally nilpotent)-by-abelian.}
\end{proof}

Finitely generated non-commutative radical rings need not be nil \cite{Smoktunowicz_Puczylowski}.
Each  nil ring is Engel \cite[Proposition 4.2]{Sysak10} and  the adjoint group $R^\circ$ of a nil algebra $R$ is locally graded  (i.e., each  finitely generated infinite subgroup of $R^\circ$ contains a proper subgroup of finite index) \cite[Theorem 1]{Sozutov_Alexandrova}. We extend this result in the following way.

\begin{corollary}  For each  Engel ($2$-torsion-free Lie solvable, $2$-torsion-free locally finite, respec\-ti\-ve\-ly) algebra $R$, its  adjoint group  $R^\circ$ is locally graded.
\end{corollary}
\begin{proof} The set  $N(R)$ is an ideal of $R$ by Corollary \ref{19} (by Theorem \ref{ZAZ}$(iii)$ and Proposition \ref{LF3}$(ii)$, respectively) and $N(R)^\circ$ is locally graded  by \cite[Theorem 1]{Sozutov_Alexandrova}. Inasmuch as $R^\circ /N(R)^\circ$ is abelian, we deduce that $R^\circ$ is locally graded.
\end{proof}

\begin{proposition} \label{47} If $R$ is a ring of bounded index, then the following statements hold:
\begin{itemize}
\item[$(i)$] $ {\mathbb P}(R)=N_r(R)$ is a locally nilpotent ring (and so ${\mathbb P}(R)^\circ$ is a locally nilpotent group);
\item[$(ii)$] if $R$ is  Engel ($2$-torsion-free locally solvable, $2$-torsion-free locally finite, respectively), then  $R^\circ$ is a (locally nilpotent)-by-abelian group;
\item[$(iii)$] if $R$ is $n$-Engel and $F(R)=0$, then $N(R)$ is nilpotent and $R^\circ$ is (torsion-free nilpotent)-by-abelian;
\item[$(iv)$] if $R$ is $n$-Engel, then $N(R)^\circ$ is (torsion  locally nilpotent)-by-(torsion-free nilpotent).
\end{itemize}
\end{proposition}
\begin{proof}  $(i)$ Obviously, $ {\mathbb P}(R)\subseteq N_r(R)$ and so  $N_r(R/{\mathbb P}(R))=0$ in view of \cite[Lemma 1.1]{Herstein69}. Thus $ {\mathbb P}(R)=N_r(R)$. Since each  nonzero homomorphic image of ${\mathbb P}(R)$ contains a nonzero nilpotent ideal,  ${\mathbb P}(R)$ is locally nilpotent. If $S=\langle g_1,\ldots, g_n\rangle$ for $ g_1,\ldots, g_n\in N_r(R)$, then $S\subseteq (\langle g_1,\ldots, g_n\rangle_{rg})^\circ$. Hence $S$ is nilpotent and  ${\mathbb P}(R)^\circ$ is  locally nilpotent.

$(ii)$ {  Inasmuch as $C(R)\subseteq N(R)={\mathbb P}(R)$ by Corollary \ref{19} (Theorem \ref{ZAZ}$(iii)$, Lemma \ref{LF2}$(iv)$, respectively), the ideal $N(R)$ is locally nilpotent in view of \cite[Lemma 1.1]{Herstein69} and the result follows.}

$(iii)$ We get that  $C(R)\subseteq {\mathbb P}(R)=N(R)$ by Corollary \ref{PP} and $N(R)^\circ$ is a torsion-free locally nilpotent $m$-Engel group for some $m\in\mathbb{N}$ in view of the part $(i)$ and \cite[Main Theorem]{Amberg_Sysak_I}. Consequently, $N(R)^\circ$ is nilpotent in view of   \cite{Zelmanov89} and   $(R/N(R))^\circ$ is abelian.

$(iv)$ The quotient group $N(R)^\circ /F(N(R))^\circ$ is torsion-free locally nilpotent $m$-Engel  for some  $m\in\mathbb{N}$ and so it is nilpotent. Futhermore, $F(N(R))^\circ$ is locally nilpotent by \cite[Corollary 2]{Amberg_Dickenschied_Sysak}.
\end{proof}

Finally, we have also the following.

\begin{proposition} \label{PA1} If the set $N(R)$ of nilpotent elements of a ring  $R$ is finite  and $C(R)\subseteq N(R)$, then the following statements hold:
\begin{itemize}
\item[$(i)$] $R$ is an $FC$-ring (i.e., the centralizer $C_R(a)=\{ r\in R\mid ra=ar\}$ is of finite index in the additive group $R^+$ for any $a\in R$ \cite{Artemovych});
\item[$(ii)$] $R^\circ$ is a finite-by-abelian group with the finite commutator subgroup (i.e., $R^\circ$ is a $BFC$-group (see e.g. \cite{Robinson})).
\end{itemize}
\end{proposition}
\begin{proof} $(i)$ The quotient group $R^+/\ker \partial_x$ is isomorphic to the image $\im \partial_x$ for any $x\in R$. Inasmuch as $\im \partial_x\subseteq C(R)$, we conclude that $R$ is an $FC$-ring.

$(ii)$ Since $N(R)^\circ$ is a finite normal subgroup of $R^\circ$ and $R^\circ /N(R)^\circ \cong (R/N(R))^\circ$, the assertion is true.
\end{proof}

\end{document}